\documentclass[12pt,leqno,twoside,a4paper]{article}
\usepackage{amsfonts,amsmath,amsthm,amssymb}
\usepackage[english]{babel}
\usepackage{color}
\newtheorem{theorem}{Theorem}[section]
\newtheorem{prop}[theorem]{Proposition}
\newtheorem{lemma}[theorem]{Lemma}
\newtheorem{corollary}[theorem]{Corollary}
\theoremstyle{definition}

\newtheorem{example}[theorem]{Example}
\theoremstyle{remark}
\newtheorem{remark}[theorem]{\bf Remark}

 \numberwithin{equation}{section}

\newcommand{\si}{\sigma}

\newcommand{\de}{\delta}

\newcommand{\Tn}{R[x;\tau ,\delta]}
\newcommand{\A}{R[x;\sigma]}

\newcommand{\C}{R[x, x^{-1};\sigma]}

\begin{document}
\title{ \bf\normalsize QUASI-DUO SKEW POLYNOMIAL RINGS}
 \author{ {\bf\normalsize  Andr\'{e} Leroy$^\ddag$,   Jerzy
Matczuk$^*$,  Edmund R. Puczy{\l}owski\footnote{ The second and third
authors were  supported
 by Polish KBN grant No. 1 P03A 032 27}} \\
\\   $^\ddag$ \normalsize Universit\'{e} d'Artois,  Facult\'{e} Jean Perrin\\
\normalsize Rue Jean Souvraz  62 307 Lens, France\\
   \normalsize  e.mail: leroy@poincare.univ-artois.fr\\
 \\ $^*$ \normalsize Institute of Mathematics, Warsaw University,\\
 \normalsize Banacha 2, 02-097 Warsaw, Poland\\
 \normalsize e.mail: jmatczuk@mimuw.edu.pl, edmundp@mimuw.edu.pl}
\date{ }
\maketitle\markboth{ \bf A.Leroy,  J.Matczuk, E.Puczy{\l}owski}
{\bf Quasi-duo rings}

\begin{abstract}
A characterization of right (left) quasi-duo skew polynomial rings
of endomorphism type and   skew Laurent polynomial rings are
given. In particular, it is shown that (1)   the polynomial ring
$R[x]$ is right quasi-duo iff $R[x]$ is commutative modulo its
Jacobson radical iff $R[x]$ is left quasi-duo, (2) the  skew
Laurent polynomial ring is right quasi-duo iff it is left
quasi-duo. These extend some known results concerning a
description of quasi-duo polynomial rings and give a partial
answer to the question posed by Lam and Dugas whether right
quasi-duo rings are left quasi-duo.
\end{abstract}

\section*{\normalsize INTRODUCTION}

An associative ring $R$ with unity is called \it right (left)
quasi-duo \rm if every maximal right (left) ideal of $R$ is
two-sided or, equivalently, every right (left) primitive
homomorphic image of $R$ is a division ring. There are  many open
problems in the area (Cf. \cite{LD}). One of the most interesting
is Problem 7.7\cite{LD}:  ``does there exist a right quasi-duo
ring that is not left quasi-duo ? ". If it were so, there would
exist right quasi-duo non-division rings, which are left
primitive. Constructing such an example might be a quite
challenging task as it is even not easy to construct left but not
right primitive rings. The known examples of such sort are based
on skew polynomial rings of endomorphism type. In that context it
is interesting to determine when such skew polynomial rings are
right (left) quasi-duo. This problem is not new and was studied by
several authors. In \cite{Y} it was proved that if $R$ is Jacobson
semisimple and the polynomial ring $R[x]$ is right quasi-duo, then
$R$ is commutative and it was expected that the same holds when
$R[x]$ is Jacobson semisimple.

 In this article  we obtain a complete
description of right (left) quasi-duo skew polynomial rings of
endomorphism type (Cf. Theorem \ref{main skew poly}) and skew
Laurent polynomial rings (Cf. Theorem \ref{Laurent q-duo}). The
descriptions  are left-right symmetric but do  not imply that
$R[x;\tau]$ is right quasi-duo iff it is left quasi-duo in
general. However the descriptions  do  imply that in order to get
a skew polynomial ring which is quasi-duo  on one side only,  one
has to begin with a coefficient ring which already has this
property. Thus skew polynomial rings of endomorphism type
themselves  do not help to construct an example of a ring which is
right but not left quasi-duo.

Theorem \ref{main skew poly} implies that the ordinary polynomial
ring $R[x]$ is right (left) quasi-duo if and only if it is
commutative modulo its Jacobson radical. This in particular shows
that $R[x]$ is right quasi-duo if and only if it is left
quasi-duo. The same result
 holds also for skew Laurent polynomial
rings.

The  article is organized as follows. In the first section we show
that the class of right quasi-duo
 rings is hereditary on certain subrings and closed under passing to  some overings.
  The most important for
 our further
studies are the results which say that the quasi-duo property is
preserved when passing from Laurent skew polynomial rings to skew
polynomial rings and from a ring with an injective endomorphism to
its Cohn-Jordan  extension.

In Section 2 we  obtain, based on a result of Bedi and Ram
\cite{BR},  a detailed description of the Jacobson radical of
right (left) quasi-duo skew polynomial rings of automorphism type,
which will be needed later on.

 The  characterization of  right (left) quasi-duo skew polynomial
rings of endomorphism type is given in  Section 3.

The last section  contains applications and examples. In
particular a characterization of quasi-duo skew Laurent polynomial
rings and   several examples delimiting the obtained results are
presented.

All rings considered in this article are associative with unity.

We frequently    consider  the endomorphism and automorphism cases
separately. Thus  we will use different letters for denoting those
maps, namely $\si$ and $\tau$ will stand for an automorphism and
an endomorphism of a ring $R$, respectively.

We will say that a subset $B$ of $R$ is \it $\tau$-stable \rm if
$\tau(B)\subseteq B$ and  $\tau^{-1}(B)\subseteq B$, where
$\tau^{-1}(B)$ denotes the preimage of $B$ in $R$. Notice that if
$I$ is an ideal of $R$ such that $\tau(I)\subseteq I$, then $\tau$
induces an endomorphism of the factor ring $R/I$ and $I$ is
$\tau$-stable iff the  induced endomorphism is injective. When
$\tau=\si$ is an automorphism of $R$, then $I$ is $\si$-stable if
$\si(I)=I$.

The skew polynomial ring  and skew Laurent polynomial ring    are
denoted by $R[x;\tau]$ and $R[x,x^{-1};\si]$, respectively. The
coefficients from $R$ are written on the left of the indeterminate
$x$.

The Jacobson radical of a ring $R$ will be denoted by $J(R)$. It
is clear from the definition, that a ring $R$ is right (left)
quasi-duo if and only if $R/J(R)$ is right (left) quasi-duo and
that $R/J(R)$ is a reduced ring in case it is right (left)
quasi-duo.  We will use these properties many times in the
article.

\section{\normalsize QUASI-DUO PROPERTY OF RINGS AND THEIR SUBRINGS}

 We start  with the following simple observation.
 \begin{lemma}\label{lemma subrings} If $R$ is a unital subring of a right quasi-duo ring
 $T$, then
every  maximal right ideal $I$ of $R$ such that  $IT\ne T$ is
two-sided.
  \end{lemma}
\begin{proof} If  $T\ne IT$, then  $IT$ is contained in a
 maximal right ideal $M$ of $T$ which is a two-sided ideal, as $T$ is a right quasi-duo ring.
Clearly $I\subseteq M\cap R$ and maximality of $I$ yields that
$I=M\cap R$, as otherwise $M\cap R=R$ and $M=T$ would follow.
\end{proof}

 A subring $S$ of a ring $T$ is called \cite{L} a  {\it corner
subring} of  $T$ if $S$ is a ring with a unity, possibly different
from that of $T$, and if there exist an additive subgroup $C$ of
$T$ such that $T=S\oplus C$ and $SC, CS\subseteq C$. The subgroup
$C$ is called a {\it complement} of $S$. We  say that $S$ is a
{\it left corner} of $T$ if the complement $C$  satisfies
$SC\subseteq C$ only.

 The most
classical example of a corner subring is a subring $S$ of the form
$eTe$, where $e\in T$ is an idempotent. Another example is the
following. Let $H$ be a monoid with an identity $e$. If
$T=\bigoplus_{g\in H} R_g $ is a $H$-graded ring, then $R_e$ is a
corner of $T$.

A natural example of a left corner is the base ring $R$ of a skew
polynomial ring $R[x;\tau,\de]$, where $\tau$ and $\de$ stand for
an endomorphism   and  a $\tau$-derivation of $R$, respectively.
Clearly $R[x;\tau,\de]=R\oplus C$, where $C=R[x;\tau,\de]x$ and
$CR\not \subseteq C$ if $\de \ne 0$.

\begin{theorem}\label{th. corners}
 Let $S$ be a left corner   of a right quasi-duo ring $T$. Then
 $S$ is also  right quasi-duo.
\end{theorem}
\begin{proof}
 Let us fix a complement $C$ of $S$ in $T$ and
  define $R=S+(1-e)\mathbb{Z}\subseteq T$, where     $e$ stands for the unity of $S$.
Since $S$ is a homomorphic image of $R$, it is enough to show that
$R$ is right quasi-duo. Let $I$ be a maximal right ideal of $R$.
If $IT\ne T$, then, by Lemma \ref{lemma subrings}, $I$ is a
two-sided ideal of $R$ . Suppose $IT=T$. Notice that $e$ is a
central idempotent of $R$,  $eR=S$ and $SC\subseteq C$.
 Using the above  we obtain $e\in eIT\cap S=(Ie(S\oplus
C))\cap S\subseteq ((I\cap S)\oplus SC)\cap S\subseteq I$, i.e.
$e\in I$. This yields that  $I$ is a two-sided ideal of $R$, as it
is the preimage of an ideal of the commutative ring $R/eR$.
\end{proof}

In the case $S$ is a corner of $T$, the above proposition was
proved by Lam and Dugas (\cite{LD}, Theorem 3.8) by completely
different arguments. From Theorem \ref{th. corners} one
immediately obtains the following corollary. The first statement
 was proved, using long arguments, in \cite{HJK}.

\begin{corollary}\label{Laurent poly}
\begin{enumerate}

  \item If $R$ is a right quasi-duo ring, then $eRe$ is also a right quasi-duo
ring, for every nonzero idempotent $e$ of $R$.
\item Suppose $H$ is
a monoid with an identity $e$. If $T=\bigoplus_{g\in H} R_g $ is a
  $H$-graded ring which is right quasi-duo, then $R=R_e$
  is a right quasi-duo ring.
  \item  Let $\tau$ be an endomorphism and $\delta$   a $\tau$-derivation of a ring $R$. If the
  skew polynomial ring
  $\Tn$ is right quasi-duo,
then $R$ is right quasi-duo.

  \end{enumerate}
\end{corollary}

Let us observe that for every positive integer $n$, the rings
$S=R[x;\tau]$ and $T=R[x,x^{-1};\si]$ are naturally  graded by the
cyclic group $G$ of order $n$ with  the $e$-components equal to
$S_0=R[x^n]\subseteq S$ and $T_0=R[x^n, x^{-n}]\subseteq T$ which
are isomorphic to $R[x;\tau^n]$ and $R[x,x^{-1},\si^n]$,
respectively. Hence Corollary \ref{Laurent poly} implies the
following:

\begin{corollary}

\begin{enumerate}\label{laurent poly 2}
\item If $R[x;\tau]$  is right
quasi-duo, then $R$ and  $R[x;\tau^n]$ are right
 quasi-duo, for every $n\geq 1$.
 \item If $\C$ is right quasi-duo, then $R$ and $R[x, x^{-1};\si^n]$ are
 right quasi-duo, for every $n\geq 1$.
\end{enumerate}

\end{corollary}

If $R$ is an Ore domain with a division ring of factions $Q$, then
$Q$ is a quasi-duo ring though  $R$ need not be quasi-duo.
 The
following theorem gives necessary and sufficient conditions for a
ring to be right quasi-duo in terms of its localization with
respect to a multiplicatively closed set generated by regular
right normalizing elements. Recall that such sets are always right
denominator sets and that an element $s\in R$ is \it right
normalizing \rm if $ Rs\subseteq sR$, i.e. the right ideal $sR$ is
two-sided.

\begin{theorem}\label{localization}
 Let $S$ be the  multiplicatively closed set generated by the set $X\subseteq
 R$ of regular right normalizing elements of $R$ such that the quotient ring $RS^{-1}$ is   right
 quasi-duo. Then the following conditions are equivalent:
\begin{enumerate}
  \item $R$ is a right quasi-duo ring;
  \item $R/xR$ is a right
 quasi-duo ring, for every  $x\in X$.
\end{enumerate}
\end{theorem}
\begin{proof}
  The  implication $(1)\Rightarrow(2)$ is clear as the class of right quasi-duo
  rings is homomorphically closed.

 $(2)\Rightarrow(1)$ Suppose $R/xR$ is right
 quasi-duo, for any $x\in X$. Let $M$ be a maximal right ideal of
 $R$.

 If $M\cap S=\emptyset$, then $MS^{-1}$ is a proper right
ideal of $RS^{-1}$ and Lemma \ref{lemma subrings} yields that $M$
is a two-sided ideal of $R$ in this case.

 Assume $M\cap S \ne \emptyset$.  Then, there are $k\geq 1$ and $x_1,\ldots ,x_k\in X$,
 such that the element  $y=x_1\cdot\ldots\cdot x_k\in M$. Since  $x_iR$,   for any $1\leq i\leq k$,
    is a two-sided ideal of $R$,  we obtain that the product of
    ideals
  $x_1Rx_2R\ldots x_kR\subseteq yR$ is contained in the
 annihilator $P$ of the simple right $R$-module $R/M$, which is a primitive
 ideal.  Therefore there exists $1\leq i\leq k$, such that $x=x_i\in P\subseteq M$.
This and the assumption imply  that   the canonical  image of $M$
in $R/xR$ is a two-sided maximal ideal of $R/xR$ and so is $M$ in
$R$.

The
 above shows that every maximal right ideal of $R$ is two-sided,
 i.e. the statement (1) holds.
\end{proof}

In the following example we construct a domain $R$ which is not
right-quasi duo but   its localization by a central element is a
duo ring.

\begin{example}
Let $p$ be an odd prime number,  $A\subseteq \mathbb{Q}$ denote
the localization of $\mathbb{Z}$ with respect to the maximal ideal
$p\mathbb{Z}$ and let $\mathbb{H}$ stand for the Hamilton
quaternion algebra over $\mathbb{Q}$ and  $R=A[i,j,k]\subseteq
\mathbb{H}$. Then $R$ is not right quasi-duo, as $R$ can be
homomorphically mapped onto $2\times 2$ matrices over the field
$\mathbb{Z}_p$. Let $S=\{p^k\mid k\geq 1\}$. Then $S$ is a central
multiplicatively closed set of $R$ such that the localization
$RS^{-1}=\mathbb{H}$ is a  duo ring.
\end{example}
 Notice that $R[x,x^{-1};\si]$ can be considered
 both as a localization of $R[x,\si]$ and of $R[x,\si^{-1}]$ with respect the multiplicatively
 closed set generated by $x$. Thus,  applying  Corollary \ref{laurent poly 2}(2) and
  Theorem
\ref{localization}  to one of the above mentioned    rings
 and $X=\{x\}$ we get the following corollary.
\begin{corollary}
\label{thm Laurent}
 If $\C$ is a right quasi-duo ring, then
 so are $\A$ and $R[x;{\si}^{-1}]$.
\end{corollary}

We will see in Example \ref{not commutative} that the reverse
implication in the above corollary does not hold in general.
However it will turn out that the polynomial ring $R[x]$ is right
quasi-duo if and only if $R[x,x^{-1}]$ is right quasi-duo.

It is known that if the endomorphism $\tau$ of the ring $R$ is
injective, then there exists a universal over-ring $A(R,\tau)$ of
$R$, called the \it $\tau$-Cohn-Jordan extension of $R$, \rm  such
that $\tau$ extends to an automorphism of $A(R,\tau)$ and
$A(R,\tau)=\bigcup_{i=0}^{\infty}\tau^{-i}(R)$.

The following proposition will enable us to replace an injective
endomorphism  by an automorphism in some of our considerations.

\begin{prop}\label{q-duo up}
 Let  $\tau$ be an injective endomorphism of $R$. Then:
\begin{enumerate}
  \item If $R$ is right (left) quasi-duo, then so is $A(R,\tau)$.
  \item If $R[x;\tau]$ is right (left) quasi-duo, then so is $A(R,\tau)[x;\tau]$.
   \end{enumerate}

\end{prop}
\begin{proof} (1) We present the proof for right quasi-duo property.
Recall (Cf. \cite{LD}) that a ring $R$ is right quasi-duo
 if and only if   $Ra+Rb=R$ implies $aR+bR=R$, for any $a,b\in
 R$. We will use this characterization in the proof.

 Let $a,b\in A(R,\tau)=A$ be such that $Aa+Ab=A$. Then, $1=za+wb$ for some $z,w\in A$ and
 by definition
 of $A$, one can pick $n\in\mathbb{N}$ such that $1=\tau^n(z)\tau^n(a)+\tau^n(w)\tau^n(b)$
 is an equality in $R$. Hence, as $R$ is right quasi-duo, we
 obtain $\tau^n(a)R+\tau^n(b)R=R$. Then also
 $\tau^n(a)A+\tau^n(b)A=A$ and $aA+bA=A$ follows as $\tau$ is an
 automorphism of $A$. This proves (1).

 (2) The injective endomorphism $\tau$ of $R$ can be extended to an
endomorphism of $R[x;\tau]  $ by setting $\tau (x)=x$. Then one
can check that $A(R[x;\tau],\tau)=A( R,\tau )[x;\tau]$. Therefore,
applying the statement (1) to the ring $R[x;\tau]$ we obtain (2).
\end{proof}

 The following example shows that the converse  of the  statement (1) in  the above
lemma does not hold in general. However, we will see later, as a
consequence of  Theorem \ref{main skew poly}, that the converse
implication in Proposition \ref{q-duo up}(2)  holds.

\begin{example} Let $\mathbb{H}$ denote  the Hamilton
quaternion algebra over the field $\mathbb{Q}$ of rationals and
$A=\mathbb{H}(x_i\mid i\in \mathbb{Z}
)=\mathbb{H}\otimes_\mathbb{Q}K$, where $K=\mathbb{Q}(x_i\mid i\in
\mathbb{Z}$ is the field of rational functions in the set
$\{x_i\mid i \in \mathbb{Z}\}$ of indeterminates. Let $\si$ denote
the $\mathbb{Q}$-automorphism of $A$ determined by conditions:
$\si|_\mathbb{H}=\mbox{id}_\mathbb{H}$ and $\si(x_i)=x_{i+1}$, for
all $i\in\mathbb{Z}$. Define $R=\mathbb{H}(x_i\mid i\geq
1)[x_0]\subseteq A$. Then  $\si$ induces an endomorphism, also
denoted by $\si$, of $R$ such that
$A=\bigcup_{i=0}^{\infty}\si^{-i}(R)$.  This means that
$A=A(R,\si)$. Clearly $A$ is a division ring, so it is a quasi-duo
ring and $R$ is not right (left) quasi-duo, as it is a polynomial
ring over a noncommutative division ring.
\end{example}

\section{\normalsize THE JACOBSON RADICAL OF QUASI-DUO SKEW POLYNOMIAL RINGS OF AUTOMORPHISM  TYPE}

 In this short section we get
  precise description of the Jacobson radical
  of a skew polynomial ring of automorphism type $\A$, in the case $\A$ is right
quasi-duo. For that we will need the following definition and
result by Bedi and Ram.

An element $a\in R$ is called $\si$-{\it nilpotent} if for every
$m\geq 1 $ there exists $n\geq 1$ such that $a{\si}^m(a)\cdots
{\si}^{mn}(a)=0$ and a subset $B$ of $R$ is called ${\si}-nil$ if
every element of $A$ is $\si$-nilpotent.

\begin{theorem}
\label{BediRam} ({\rm \cite{BR}, Theorem 3.1}) If $\si$ be an
automorphism of a ring $R$, then there exist  $\si$-nil ideals
$K\subseteq J(R)$ and $I$ of $R$ such that
  $J(\A)=(I\cap J(R) )+I[x;\si]x$ and
    $J(\C)=K[x,x^{-1};\si]$.
\end{theorem}

\begin{lemma}\label{basic fact}
Suppose that $\A$ is right quasi-duo. Let  $M$ be a maximal ideal
of $\A$ and  $M_0=M\cap R$. Then  precisely one of the following
conditions holds:
\begin{enumerate}
\item If there exist $n\geq 1$ such that $x^n\in M$, then $x\in M$, $M_0$
is a maximal ideal of $R$ and  $M=M_0+\A x$;
 \item If for every $n\geq 1$, $x^n\not\in M$, then  $M$ and $M_0$
are   $\si$-stable, $M_0=\{r\in R\mid \exists _{m\geq 0}\; rx^m\in
M\}=\{r\in R\mid rx\in M\}$ and $A=R/M_0$ is a domain such that
$A[x;\si]$ is right quasi-duo.
\end{enumerate}
\end{lemma}
\begin{proof}The $n$-th power $(x)^n$ of the ideal generated by
$x$ is equal to $\A x^n$. Thus if $x^n\in M$, then $x\in M$ and
clearly $M=M_0+\A x$, i.e. (1) holds.

Suppose $x\not \in M$ and let $m\in M$. Then $xm=\si(m)x$ and
$mx=x\si^{-1}(m)$ belong to $M$.   Since $M$ is prime ideal and
$x$ is a centralizing element, we obtain $\si(m), \si^{-1}(m)\in
M$, for $m\in M$, i.e. $M$ is $\si$-stable. A similar argument
leads to $M_0=\{r\in R\mid \exists _{m\geq 0}\; rx^m\in M\}$. Now
$B=R/M_0$ embeds canonically into $\A/M$ which is a division ring
(as $\A$ is right quasi-duo), so $B$ is a domain and $B[x;\si]$ is
right quasi-duo as a homomorphic image of $\A$. Hence (2) holds.
\end{proof}
Let us remark that the ideal $M_0$ from Lemma \ref{basic fact}(1)
need not be maximal in $R$. For example, if $K$ is a field, then
the ideal $M=(tx-1)$ of $K[[t]][x]$ is maximal and $M_0=0$ is not
maximal in the base ring $K[[t]]$.

Although  only automorphisms are considered  in this section, we
will formulate  the following definition in more general setting.
For a ring $R$ with an endomorphism $\tau$ we set
$$N(R)=\{a\in R\mid \exists_{n\geq 1}\;\;
a\tau(a)\ldots\tau^n(a)=0\}.$$ The definition of $N(R)$ depends on
the choice of the endomorphism $\tau$, we hope that the proper
choice of $\tau$ will be clear from the context. Notice that the
set $N(R)$ is $\tau$-stable

\begin{prop}\label{radical} Let $\si$ be an automorphism of $R$
and  $\A$ be right (left) quasi-duo ring. Then:
\begin{enumerate}
\item  $N(R)$ is a two-sided, $\si$-stable ideal of $R$ such that
$J(\A)= (N(R)\cap J(R))+ N(R)[x;\si]x$, the ring $R/N(R)$ does not
contain nonzero $\si$-nilpotent elements and the ring   $R/(N\cap
J(R))$ is reduced. In particular, every nilpotent element of $R$
is $\si$-nilpotent.

\item Let $\mathcal{A}$ denote the set of
all maximal ideals $M$ of $\A$ such that $x\not\in M$ and
$\mathcal{B}$ the set of remaining maximal ideals of $\A$. Then:

\begin{enumerate}
\item $\bigcap_{M\in \mathcal{B}}M=J(R)+R[x;\si]x$;
\item $N(R)=\bigcap_{M\in
\mathcal{A}}M_0$, where $M_0=M\cap R$ and $\bigcap_{M\in
\mathcal{A}}M=N(R)[x;\si]$;
\end{enumerate}
\end{enumerate}
\end{prop}
\begin{proof}(1) Let $I$ be the ideal of $R$ described in  Theorem \ref{BediRam}.
 Clearly $I\subseteq N(R)$, as  $I$ is  $\si$-nil.
Since $\A / J(\A)$ is a semiprimitive one-sided quasi-duo ring, it
is reduced. If $a\in N(R)$, then $ax\in \A$ is a nilpotent element
and $N(R)\subseteq I$ follows, i.e. $I=N(R)$. Now it is easy to
complete the proof of (1) using Theorem \ref{BediRam} and the fact
that $\A /J(\A)$ is a reduced ring.

(2)(a) If $M_0$  is a maximal ideal of $R$, then $M_0+\A x\in
\mathcal{B}$.  This implies, as $R$ is a right quasi-duo, that
$\bigcap_{M\in \mathcal{B}}M=J(R)+\A x$, i.e. (2)(a) holds.

(b) Using Lemma \ref{basic fact} and the statement (a) we obtain
\begin{equation*}
\begin{split}
 N(R)  & =\{a\in R\mid ax\in J(\A)\}=\bigcap_{M\in\mathcal{A} \cup
\mathcal{B}}\{a\in R\mid ax\in
M\}\\
& =\bigcap_{M\in\mathcal{A}}\{a\in R\mid ax\in M\}=
\bigcap_{M\in\mathcal{A}}M_0.
\end{split}
\end{equation*}
This shows that  the first equation from (b) holds. The second one
is a consequence od (1) and (a) above.
\end{proof}

The above proposition offers precise  description of  $J(\A)$
provided $\A$ is right (left) quasi-duo. Later on,  as a side
effect of the main Theorem \ref{main skew poly},  we will obtain a
similar characterization of the Jacobson radical of a right (left)
quasi-duo skew polynomial ring $R[x;\tau]$ of endomorphism type.
\section{\normalsize MAIN RESULTS}
In this section we  characterize skew polynomial rings which are
right (left) quasi-duo. When $\tau$ is an automorphism, then the
opposite ring to $R[x;\tau]$ is also a skew polynomial ring of
automorphism type. This means that the characterizations  of right
and left quasi-duo properties of  skew polynomial rings are
left-right symmetric, in this case. We  show in Theorem \ref{main
skew poly} that we also have such symmetry for arbitrary
endomorphisms. This will be achieved by reducing the general case
to the case of skew polynomial rings of automorphism type.

Recall that a ring $R$ with an automorphism $\si$ is called \it
$\si$-prime \rm if the product $IJ$ of nonzero $\si$-stable ideals
$I,J$ of $R$ is always nonzero and a $\si$-stable ideal $P$ of $R$
is $\si$-prime if the ring $R/P$ is $\si$-prime.  By $P_{\si}(R)$
we will denote \it the $\si$-pseudoradical of $R$, \rm that is
$P_{\si}(R)$ is the intersection of all non-zero $\si$-prime
ideals  of $R$ (by definition, the empty intersection is equal to
$R$).

The commutator  $ab-ba$ of elements $a$ and $b$ from a ring $R$
will be denoted by $[a,b]$.

In the sequel we  will need the following two lemmas, the first
one is a special case of Lemma 3.2(1) from \cite{Matczuk2}.
\begin{lemma}
\label{3.2 max} Let $R$ be a domain.  Suppose that $R[x;\si]$
contains a maximal ideal $M$ such that $M\cap R=0$.   Then
$P_{\si}(R)\ne 0$.
\end{lemma}
\begin{lemma} \label{subdirect} Let $\tau$ be an endomorphism of a ring $R$
and $U\subseteq V$ be ideals of $R$. If $\tau(V)\subseteq V$, then
$U+V[x;\tau]x$ is a two-sided ideal of $R[x;\tau]$ and the ring
$R[x;\tau]/(U+V[x;\tau]x)$ is right (left) quasi-duo if and only
if $R/U$ and $R[x;\tau]/V[x;\tau]\simeq (R/V)[x;\tau]$ are right
(left) quasi-duo rings.
\end{lemma}
\begin{proof}
It is clear that $U+V[x;\tau]x $ is a two-sided ideal of
$R[x;\tau]$ and that $\tau$ induces an endomorphism, also denoted
by $\tau$, of the ring $R/V$. Notice that $(U+R[x;\tau]x)\cap
V[x;\tau]=U+V[x;\tau]x$. This implies that $R[x;\tau]$ is a
subdirect product of rings $R/U$ and $(R/V)[x;\tau]$. Now it is
easy to complete the proof with the help  of   Corollary 3.6(2)
\cite{LD}, which states that a finite subdirect product of right
(left) quasi-duo rings is also such.
\end{proof}

The following theorem is the key ingredient for our  description
of one-sided quasi-duo skew polynomial rings.

\begin{theorem}\label{R domain skew poly}
 Let $R$ be a domain with an automorphism $\si$.
  If $R[x;\si]$ is right quasi-duo, then  $R$ is commutative and
$\si=\mbox{\rm id}_R$.
\end{theorem}
\begin{proof} Suppose that $\A$ is right quasi-duo. Let $\mathcal{A}$ denote
the set of all maximal ideals $M$ of $ R[x;\si] $ such that
$x\not\in M $. Since $R$ is a domain and $\A$ is right quasi-duo,
Proposition \ref{radical} implies that:\\
 (i) $0=J(\A)=\bigcap_{M\in \mathcal{A}} M$,
  i.e. $\A$ is a subdirect product of division rings $\A/M$, where $M$ ranges over
$\mathcal{A}$.\\
(ii) $M_0=M\cap R$ is a $\si$-stable ideal of $R$,  for any $M\in
\mathcal{A}$.

 Let $M\in \mathcal{A}$ and $\pi\colon
R[x;\si]\rightarrow R[x;\si]/M=D $ be an epimorphism onto a
division ring $D$. We claim that $D$ is commutative. Since
 $M_0$ is a $\si$-stable ideal of $R$,
$\pi$ induces an epimorphism $f\colon (R/M_0)[x;\si]\rightarrow D$
such that  $\ker f\cap(R/M_0) =0$. Moreover $x\not\in\ker f$, as
$x\not\in M$. Since $D$ is a division ring, the above yields that
 $R/M_0$ is a domain,
$\ker f$ is a maximal ideal of $(R/M_0)[x;\si]$ not containing
$x$. Thus, eventually replacing $R$ by $R/M_0$ we may assume that
$M\in \mathcal{A}$ has zero intersection with the base ring $R$.
Then, Lemma \ref{3.2 max}, shows that $P_{\si}(R)\ne 0$.

Take $ 0\ne a\in  P_{\si}(R)$    and consider a maximal right
ideal $W$ of $R[x;\si]$ containing $ax+ 1$. By assumption, $W$ is
a two-sided ideal.  Clearly $x\not \in W $ and, by Lemma
\ref{basic fact}, $W_0=W\cap R$ is a (completely) prime,
$\si$-stable ideal of $R$.  Thus, if $W_0$ would be nonzero it
would contain $P_{\si}(R)$ and $W=\A$ would follow. Therefore
$W_0=0$.

Observe that $[a,ax+1]=(a^2-a\si(a))x\in W$, as $W$ is a two-sided
ideal of $\A$.  Since  $x$ is a normalizing element of $\A$ and
$W$ is a prime ideal, we obtain $a(a-\si(a))\in W\cap R=0$. This
proves that  $a(a-\si(a))=0$, for any $a\in P_\si(R)$ and the fact
that $R$ is a domain implies that $\si$ is identity on
$P_{\si}(R)$. Then for arbitrary $r\in R$, we have $ar\in
P_{\si}(R)$ and $0=\si(ar)-ar=a(\si(r)-r)$ and  $\si=\mbox{id}_R$
follows.

Let $b\in R$. Then, as $\si=\mbox{id}_R$, we have
$[b,a]x=[b,ax+1]\in W$ and, similarly as above, we get $[b,a]=0$
for all  $b\in R$. This means that  $0=[b,ar]=a[b,r]$, for all
$b,r\in R$,  and yields commutativity of the domain $R$. This
shows that $D$ is commutative as a homomorphic image of a
commutative ring $\A=R[x]$ and completes the proof of the theorem.
\end{proof}

For the endomorphism  $\tau$ of  a ring $R$ let us set
$\mathcal{K}=\bigcup_{i=1}^\infty \ker \tau^i$. Then, as
$\mathcal{K}$  is $\tau$-stable,  $\tau$ induces an injective
endomorphism, also  denoted by $\tau$, of the ring $\bar
R=R/\mathcal{K}$. In particular,  we can consider the
$\tau$-Cohn-Jordan extension $A(\bar R ,\tau)$ of $\bar R$.

Keeping the above notation we can formulate the following theorem
which  characterizes   one-sided quasi-duo skew polynomial rings
of endomorphism type.

\begin{theorem} \label{main skew poly}
 For a ring $R$ with an endomorphism $\tau$, the following
 conditions are equivalent:
\begin{enumerate}
\item $R[x;\tau]$ is a right (left) quasi-duo ring;

\item $R$ and $\bar R[x;\tau]$ are right (left)
quasi-duo rings;

\item $R$ and $A(\bar R,\tau)[x;\tau]$ are right (left) quasi-duo
rings;
\item The following conditions hold:
\begin{enumerate}
 \item $R$ is right (left) quasi-duo and $J(R[x;\tau])=(J(R)\cap
 N(R) )+N(R)[x;\tau]$;
 \item $N(R)$ is a $\tau$-stable ideal of $R$, the factor ring   $R/N(R)$ is
  commutative  and the endomorphism  $\tau$
induces  identity on $R/N(R)$.
\end{enumerate}
\end{enumerate}

\begin{proof} It is clear that (1) implies (2),  as the rings in (2)
are homomorphic images of $R[x;\tau]$.

The implication  $(2)\Rightarrow (3)$ is given by  Proposition
\ref{q-duo up}(2) applied to the ring $\bar R$.

$(3)\Rightarrow (4)$ Case 1, when $\tau=\si$ is an automorphism of
$R$, i.e. $R=\bar R = A(R,\tau)$. Suppose  that $R$ and $\A$ are
right quasi-duo.
 Proposition \ref{radical}  shows that  the statement (4)(a) holds, in
 this case.

  Let $\mathcal{A}$ denote the set  of all maximal ideals of $\A$
 such that $x\not \in M$. Proposition \ref{radical} together with Lemma \ref{basic
 fact}(2) yield that $N(R)$ is a $\si$-stable ideal of $R$ and
 the ring $(R/N(R))[x;\si]\simeq \A/(N(R)[x;\si])$
 is a subdirect product of  domains $(R/M_0)[x,\si]\simeq
 \A/(M[x;\si])$, where $M$ ranges over $\mathcal{A}$   and $M_0=M\cap R$. Now
 the statement (4)(b) is a consequence of Theorem \ref{R domain skew
 poly}. Therefore  Case 1 holds for right quasi-duo property of $\A$.
  Notice  that we also have proved that the ring $R/N(R)$ is reduced.

Making use of  the isomorphism $(\A)^{op}\simeq
R^{op}[x;{\si}^{-1}]$, where $T^{op}$ denotes the ring opposite to
a ring $T$, one can easily get the left version of Case 1.

Case 2, when $\tau$ is an arbitrary endomorphism of $R$. Recall
that for a ring $B$ with an endomorphism $\tau$, $N(B)=\{b\in
B\mid  \exists_ {n\geq 1} b\tau(b)\ldots\tau^n(b)=0  \}$.  Let
$A=A(\bar R,\tau)$, where $\bar R=R/\mathcal{K}$. It is easy to
check that $ {\mathcal K}\subseteq N( R)$, $N(\bar R)=N(A)\cap
{\bar R}$ and $N({\bar R})=N(R)/{\mathcal K}$. This implies that
$N(R)$ is the kernel of the composition of natural homomorphisms
$R\rightarrow \bar R\rightarrow \bar R/N(\bar R) \hookrightarrow
A/N(A)$. The above and  Case 1 applied to the ring $A$ imply that
$N(R)$ is a $\tau$-stable ideal of $R$, $R/N(R)$ is a commutative
reduced ring as $A/N(A)$ is such and $\tau$ induces identity on
$R/N(R)$, as $\tau$ induces identity on $A/N(A)$. This means that
the statement (4)(b) holds and also proves that
$J(R[x;\tau])\subseteq N(R)[x;\tau]$, since the factor ring
$R[x;\tau]/(N(R)[x;\tau])\simeq (R/N(R))[x]$ is semiprimitive as
$R/N(R)$ is reduced. In fact, the last inclusion together with the
fact that $J(R[x;\tau])\cap R\subseteq J(R)$ imply that
$J(R[x;\tau])\subseteq (J(R)\cap N(R))+ N(R)[x;\tau]x$.

 For every ${\bar r}\in N(\bar R)$, ${\bar
r}x$ is a nilpotent element of ${\bar R}[x;\tau]=(R/\mathcal K
)[x;\tau]$. Hence, since $A[x;\tau]$ is a right (left) quasi-duo
ring, $N(\bar R)[x;\tau]x\subseteq J(A[x;\tau])$. Thus for every
$f(x)\in N(\bar R)[x;\tau]x$ there exists $g(x)\in A[x;\tau]$ such
that $(1-f(x))(1+g(x))=1$. On the other hand
$(1-f(x))(1+f(x)+(f(x))^2+\cdots)=1$ in the power series ring
$A[[x;\tau]]$. Consequently $g(x)=f(x)+(f(x))^2+\cdots \in N(\bar
R)[x;\tau]$, so $(N(\bar R)[x;\tau]x)\subseteq J({\bar
R}[x;\tau])$. Therefore, as  ${\mathcal K}[x;\tau]x$ is a nil
ideal of $R[x;\tau]$ and $N(\bar R)=N(R)/{\mathcal K} $, we get
that $N(R)[x;\tau]x\subseteq J(R[x;\tau])$. Now, with the help of
this inclusion, one can check that every element of the ideal
$J=(J(R)\cap N(R))+ N(R)[x;\tau]x$ is quasi regular in
$R[x;\tau]$.
 This together with the
above proved inclusion $J(R[x;\tau])\subseteq J$ show  that
$J(R[x;\tau])=(J(R)\cap N(R) )+N(R)[x;\tau]$ and completes the
proof of the implication $(3)\Rightarrow (4)$.

 The implication $(4)\Rightarrow (1)$ is a direct consequence of
 Lemma \ref{subdirect}
 with  $U=J(R)\cap N(R)$ and $V=N(R)$.
\end{proof}
\end{theorem}
\begin{remark}

1. Statements (2) and (3) of the above theorem are of rather
technical nature, but are important   steps  for obtaining the
equivalence of (1) and (4). When $\tau$ is an automorphism of $R$,
then both (2) and (3) boil down to the statement (1), as $A(\bar
R;\tau)=R=\bar R$ in this case.\\
2. Note that  $R[x;\tau]$ is a commutative ring if and only if $R$
 is commutative and $\tau =\mbox{id}_R$.
Thus the  condition (4)(b) can be equivalently stated as ``$N(R)$
is a $\tau$-stable ideal and the ring
$R[x;\tau]/(N(R)[x;\tau])\simeq (R/N(R))[x;\tau]$ is
commutative''.  One can
reformulate some other results in the article in a similar way.\\
3. Observe that it is shown in the proof of Theorem \ref{main skew
poly}  that if $R[x;\tau]$ is right (left) quasi-duo, then the
commutative ring  $R/N(R)$ has to be reduced.
\\ 4. Theorem \ref{main skew poly} shows that the use of skew
polynomial rings of endomorphism type can not be effective in
constructing a ring which is quasi-duo only on one side, as for
that aim  it would be necessary to construct   the base ring with
the same properties first.
\end{remark}

\section{\normalsize APPLICATIONS AND EXAMPLES}

It is known (\cite{R}, \cite{BR}) that if either $R$ is one sided
noetherian or the automorphism $\si$ is of locally finite order,
i.e. when for any $a\in R$ there is $n\geq 1$ such that
$\si^n(a)=a$, then $J(\A)=I[x;\si]$, for some nil ideal $I$ of
$R$. Since the Jacobson radical of quasi-duo rings contains all
nilpotent elements, Theorem \ref{main skew poly} gives immediately
the following theorem.

\begin{theorem}\label{finite case rad} Suppose that $\si$ is an
automorphism of $R$ and either $R$ is  one-sided noetherian  or
$\si$ is of locally finite order. Then $\A$ is right  quasi-duo if
and only if $J(\A)={\mathcal N}(R)[x;\si]$, where ${\mathcal
N}(R)$ is the nil radical  of $R$, $R/{\mathcal N}(R)$ is
commutative and the automorphism of $R/{\mathcal N}(R)$ induced by
$\si$ is equal to $\mbox{\rm id}_{R/{\mathcal N(R)}}$.
 \end{theorem}

The following theorem characterizes quasi-duo skew Laurent
polynomial ring.

\begin{theorem}\label{Laurent q-duo} The skew Laurent polynomial
ring $R[x,x^{-1};\si]$ is right (left) quasi-duo if and only if
$J(R[x,x^{-1};\si])=N(R)[x,x^{-1};\si]$, $R/N(R)$ is commutative
and the automorphism of $R/N(R)$ induced by $\si$ is equal
$id_{R/N(R)}$ if and only if $\C/J(\C)$ is commutative.
\end{theorem}

\begin{proof} Let $T=\C$. We will prove the only nontrivial
implication, that is, we will  show that if $T$ is one-sided
quasi-duo, then $J(T)$, $N(R)$ and $\si$ satisfy conditions
described in the theorem. To this end suppose $T$ is right
quasi-duo.  By Theorem \ref{BediRam},
$J(R[x,x^{-1};\si])=K[x,x^{-1};\si]$ for a suitable $\si$-nil
ideal $K$ of $R$. Clearly $K\subseteq N(R)$. Since $N(R)x$
consists of nilpotent elements and $T/J(T)$ is reduced,
$N(R)\subseteq J(T)$ and $J(T)=N(R)[x,x^{-1};\si]$ follows,  in
this case. By Corollary \ref{thm Laurent}, $R[x;\si]$ is right
quasi-duo and  Theorem \ref{main skew poly} shows that $R/N(R)$ is
commutative and  the automorphism of $R/N(R)$ induced by $\si$ is
equal $id_{R/N(R)}$.
\end{proof}

It is known \cite{KS} that for every ring $R$, $J(R[x,x^{-1}])\cap
R[x]=J(R[x])$. Thus, Theorems \ref{finite case rad} and
\ref{Laurent q-duo} give the following:

\begin{corollary} \label{Polynomials + laurent}
  $R[x]$ is right (left) quasi-duo if and only if $R[x,x^{-1}]$ is
  right (left)
  quasi-duo if and only if
    $J(R[x])=\mathcal{N}(R)[x]$ and the factor ring $R/\mathcal{N}(R)$ is
  commutative, where $\mathcal{N}(R)$ denote the nil radical of
  $R$.
 \end{corollary}

\begin{remark} If the polynomial ring $R[x,y]$ in two commuting indeterminates
$x, y$ is right-quasi-duo, then  $R[x]$ is also right quasi-duo as
a homomorphic image of $R[x,y]$. However the converse does not
hold. Indeed, in \cite{PS} a nil ring $C$ was constructed such
that $C[x]$ is Jacobson radical but not nil.
 Let $R$ be the ring obtained from $C$ by adjoining
unity by the ring $\mathbb{Z}$ of integers. Then $J(R[x])=C[x]$
and $R[x]/J(R[x])\simeq \mathbb{Z}[x]$, so $R[x]$ is quasi-duo.
 However
$R[x]/{\mathcal N}(R[x])$ is not commutative as otherwise $C$
would be contained in ${\mathcal N}(R[x])$, which is impossible.
Hence Corollary \ref{Polynomials + laurent} shows that the ring
$R[x,y]\simeq R[x][y]$ is not right (left) quasi-duo.
\end{remark}
The aim of the following example is twofold.
 Firstly,  it   shows that there exists a ring $R$ with an automorphism
 $\si$ such that  both $\A$ and
 $R[x;\si^{-1}]$ are right quasi-duo but $\C$ is not, i.e. the
 converse of  Corollary \ref{thm Laurent} does not hold.
 Secondly, contrary to the cases of polynomial and  skew Laurent
 polynomial rings, there exist right quasi-duo skew polynomial
 ring $\A$ such that $\A/J(\A)$ is not commutative.
\begin{example}\label{not commutative} Let $\Delta$ be a right quasi-duo ring which is
Jacobson semisimple (for example one can take $\Delta$ to be a
division ring) and $P=\prod_{i\in \mathbb{Z}}\Delta_i$ be
 the direct product of   $\Delta_i=\Delta$. We can identify the
center $F$ of $\Delta$ with the subring of $P$
 consisting with
  constant sequences
 $(f)_{i\in \mathbb{Z}}$, $f\in F$.
 Let $R=F+I \subseteq
 P$, where $I=\bigoplus_{i\in\mathbb{Z}} \Delta_i$. Let
 $\si$ be the right shifting automorphism of
$R$ given by
$\si((a_i)_{i\in\mathbb{Z}})=(a_{i+1})_{i\in\mathbb{Z}}$.

One can easily check $R$ is right quasi-duo. Observe also that
$I[x;\si]x$ and $I[x;\si^{-1}]x$ are nil ideals of $\A$ and
$R[x;\si^{-1}]$, respectively, and the factor rings
$\A/(I[x;\si])$ and $R[x;\si^{-1}]/I[x;\si^{-1}]$ are isomorphic
to the commutative polynomial ring  $F[x]$. Now Lemma
\ref{subdirect} yields that both $\A$ and $R[x;\si^{-1}]$ are
right quasi-duo. Notice that $N(R)=I$ and, as $J(\Delta)=0$, the
ring  $R$ is Jacobson semisimple.
  Thus, by Theorem \ref{BediRam}, $J(R[x,x^{-1};\si])=0$
and $J(\A)=I[x;\si]x$. Now, Theorem \ref{Laurent q-duo} implies
that $R[x,x^{-1};\si]$ is not right quasi-duo. Notice also that if
$\Delta$ is not commutative then $\A/J(\A)$ is  not commutative as
well.
\end{example}
The following example shows  that it is possible that $\A$ is
right quasi-duo whereas $R[x;\si^{-1}]$ is not.
\begin{example}\label{ex. localization}
 Let $K$ be a field,
  $R=K\langle x_i\mid i\in \mathbb{Z},\; x_ix_j=0, \mbox{ for } i < j\rangle$
  and $I$ denote the ideal of $R$ generated
  by all products $x_kx_l$, where  $k,l\in \mathbb{Z}$ and $k>l$.
   Let $\si$ be the $K$-automorphism of $R$ given by
  $\si(x_i)=x_{i+1}$, for all $i\in \mathbb{Z}$.

Note that  $\A/I[x;\si]$ is isomorphic to $P[x;\si] $, where
$P=\prod_{i\in\mathbb{Z}}\Delta _i$ and $\Delta_i=K[x_i]$, for
$i\in \mathbb{Z}$. The same arguments as in Example \ref{not
commutative} (with $F$ replaced by $\prod_{i\in \mathbb{Z}}K$),
show  that this factor ring is right quasi-duo. Thus, for proving
that $\A$ is right quasi-duo, it is enough to show that
$I[x;\si]\subseteq J(\A)$. To this end, observe that for any
monomial $m\in I$ we have $m  \A m=0$ in $ \A$. This implies
$m\A\subseteq J(\A)$, for all monomials $m\in I$ and hence
$I\subseteq J(\A)$ follows.

Finally the ring $R[x;\si^{-1}]$ is not right quasi-duo as
otherwise, by Theorem \ref{main skew poly}, $N$ would be an ideal
of $R$ and $R/N$ would be commutative, where $N=N(R)$ is taken
with respect to $\si^{-1}$. However $x_1x_0^2\ldots
x_{-n+1}^2x_{-n}\ne 0$, for any $n\geq 1$, i.e. $x_1x_0\not\in N$
and clearly $0=x_0x_1\in N$.
\end{example}
Corollary \ref{thm Laurent} shows that for the ring $R$ from the
above example the ring $\C$ is also not right quasi-duo.

 Note that   in Examples \ref{not commutative} and \ref{ex. localization}
 the ring  $R$
contains maximal ideals with infinite orbits under the action of
$\si$. We close the article with
  a result showing that  this is one of the reasons for which
$R[x,x^{-1};\si]$ is not quasi-duo.

\begin{theorem}\label{orbits} Let $\si$ be an automorphism of $R$.
For any maximal ideal $M$  of $R$ we have:
\begin{enumerate}
\item Suppose the ring $\A$ is
right quasi-duo. Then
 either $\si(M)=M$ or $M$ has infinite orbit under the
action of $\si$. Moreover, if
    $\si(M)=M$, then  the ring $R/M$ is a field and $\si$ induces identity on $R/M$.

\item Suppose the ring $\C$ is right quasi-duo. Then $\si(M)=M$, the ring $R/M$
 is a field and $\si$ induces identity on $R/M$.
\end{enumerate}
\end{theorem}
\begin{proof} Let $M$ be a maximal   ideal of $R$.  Notice that
$R/M$ is a division ring,  as $\A$ is right quasi-duo.

 (1)(a) Suppose that $M$ has finite orbit  $\{M,\si(M),\dots , \si^n(M)\}$
 under the action of $\si$ and let $I=M\cap\si(M)\cap\dots \cap
 \si^n(M)$. Then, $\si$ induces an automorphism  of $\bar{R}=R/I$,
 which is also denoted by $\si$. Notice that   $\bar{R}$ is isomorphic to the direct
 product of division rings $R/\si^k(M)$, $0\leq k\leq n$, and $\si$ induces its
 isomorphism permuting the components. It is clear that the set of
 $\si$-nilpotent elements of $\bar R$ forms an ideal if and only
 if $\si (M)=M$, so the first part of (1)  is a consequence of Proposition \ref{radical}.

   Suppose now that $\si(M)=M$. Then $\si$ induces an automorphism of the
 division ring  $R/M$ and $R/M[x;\si]$ is right quasi-duo. Now the
 remaining part of (1) is a consequence of Theorem \ref{R domain skew poly}.

  (2)  Since $\C$ is
 right quasi-duo, there exists a maximal ideal of $\C$, say $W$,
 such that  $M\C\subseteq W$.
 Notice that maximality of $M$ implies that $W\cap R=M$.
 We also have   $M+xMx^{-1}=M+\si(M)$ and $M+\si^{-1}(M)=
 M+x^{-1}Mx$ are contained in $W\cap R=M$.
 This shows that $\si(M)=M$.

 By Corollary \ref{thm Laurent},  the ring $\A$ is also a right quasi-duo
   and with the help of the statement (1) one can easily
 complete the proof of the theorem.
\end{proof}

Concluding, in this article we got a complete and, as it seems,
quite satisfactory characterization of quasi-duo skew polynomial
rings of endomorphism type.  Our results show that it is hard  to
expect that such rings might be helpful in constructing right but
not left quasi-duo rings.  It seems that more appropriate for that
aim could be  skew polynomial rings $R[x;\tau,\de]$, where $\de$
denotes a skew $\tau$-derivation of $R$. The situation reminds a
bit the case of duo property of rings (recall that a ring $R$ is
called right (left) duo if every right (left) ideal of $R$ is
two-sided). It is known (Cf.\cite{Marks}) that $R[x;\tau]$ is
right duo iff it is left duo, whereas there exists
(Cf.\cite{Matczuk}) a skew polynomial ring $R[x;\tau,\delta]$ such
that it is right but not left duo. Note however that the examples
constructed in \cite{Matczuk} are left and right quasi-duo. In
that context it would be interesting to characterize right and
left quasi-duo property of skew polynomial rings $R[x;\tau,\de]$.

\end{document}